\documentclass[11pt]{article}

\usepackage[utf8]{inputenc}
\usepackage{amsmath}
\usepackage{amssymb}
\usepackage{amsthm}
\usepackage{bm}
\usepackage{xcolor}
\definecolor{slate}{RGB}{71,85,105}
\definecolor{amber}{RGB}{217,119,6}
\usepackage{hyperref}
\usepackage{tikz}
\usetikzlibrary{shapes,arrows,positioning,calc}

\theoremstyle{definition}
\newtheorem{definition}{Definition}[section]

\newtheorem{remark}[definition]{Remark}

\theoremstyle{plain}
\newtheorem{theorem}[definition]{Theorem}
\newtheorem{lemma}[definition]{Lemma}
\newtheorem{corollary}[definition]{Corollary}
\newtheorem{proposition}[definition]{Proposition}

\title{\textbf{The Octonionic Rational Wave Numbers: Non-Associative Sequence Algebras, Moufang Loops, and Exceptional $G_2$ Symmetries}}
\author{\textbf{Terence R. Smith}\thanks{Email: \texttt{TerenceRSmith55@gmail.com}}}
\date{\today}

\begin{document}

\maketitle

\begin{abstract}
We generalize discrete rational wave numbers (arXiv:2609.23233) to the octonionic rational wave numbers $\mathbb{O}^\mathbb{Z}$ over Cayley's octonions 
$\mathbb{O}$. Rather than assuming higher imaginary units, all rational wave number spaces $\mathbb{C}^\mathbb{Z}, \mathbb{H}^\mathbb{Z}, \mathbb{O}^\mathbb{Z}$ are derived constructively from two scalar primitives: 
the coordinate advance $E_Z$ and phase calibration $E_{\bar{1}}$. 
Their closure under pointwise multiplication, inversion, and rational roots generates an infinite divisible Moufang loop on the sequence ring $\mathcal{S}_\mathbb{O} := \mathbb{O}^\mathbb{Z}$, terminating strictly at $\mathbb{O}^\mathbb{Z}$ via sedenionic zero-divisor excision. By Artin's Theorem, pairwise compositions remain strictly associative, whereas three-wave interactions generate an intrinsic non-associative curvature quantified by the associator $[u, v, w] = (uv)w - u(vw)$. We define a real-valued sequence inner product that bypasses the Jordan--von Neumann--Wigner obstruction, establishing the orthonormality of the seven Fano coordinate primitives. We establish three principal theorems: (1) The Invariant Hexagonal Mesh Theorem, proving that cumulative integration exhibits a scale-invariant bifurcation at $n=6$ governed by the root system of $G_2 = \operatorname{Aut}(\mathbb{O})$ and its $\mathrm{SU}(3)$ stabilizer subgroup; (2) The Algebraic Phase Jitter Theorem, showing that Catalan parenthesization trees across non-Fano triads produce an irreducible zero-point phase fluctuation, yielding an algebraic origin for vacuum jitter; and (3) The Rational Loop Closure Theorem, proving that rational frequencies select closed finite Moufang sub-loops, while irrational frequencies yield topologically dense, non-closing trajectories. This provides an autonomous, discrete algebraic foundation for non-associative geometry and fundamental physical symmetries.
\end{abstract}

\noindent\textbf{Keywords:} Octonions, Non-associative rings, Rational wave numbers, Moufang loops, Exceptional Lie group $G_2$, Fano plane, Catalan trees, Division algebras, Invariant lattices. \\
\textbf{MSC 2020:} 17D05, 17A35, 20N05, 11B83, 22E60, 81R05.

\vspace{1em}

\section{Introduction}\label{sec:intro}

We extend the algebraic classification of discrete rational wave numbers on a one-dimensional spatial lattice \(\mathbb{Z}\)—established for complex and quaternionic lattice fields \(\mathbb{C}^{\mathbb{Z}}\) and \(\mathbb{H}^{\mathbb{Z}}\) in \cite{smith2025closure} and \cite{smith2026rational}—to the non-associative octonionic field space \(\mathbb{O}^{\mathbb{Z}}\). These three spaces comprise lattice sequences taking values in the norm division algebras \(\mathbb{C}\), \(\mathbb{H}\), and \(\mathbb{O}\), respectively.

Starting from a minimal set of two coordinate primitives:
\begin{equation}
E_Z[m] = \exp(2\pi i m) \equiv 1, \qquad E_{\bar{1}}[m] = \exp(2\pi i) \equiv 1, \quad \forall m \in \mathbb{Z},
\end{equation}
that are lifted into an unwrapped affine phase space $\mathbb{R}^\mathbb{Z} / \mathbb{Z}^\mathbb{Z}$, the linear rational wave closure space $\mathcal{G}$ under pointwise multiplication $(\otimes)$, multi-branch rational roots $(\cdot)^{1/n}$, and inversion $(\cdot)^{-1}$ was classified by Smith (Theorem~3.3 of \cite{smith2025closure}) as an infinite divisible abelian torsion group:
\begin{equation}
\mathcal{G} \cong (\mathbb{Q}/\mathbb{Z}) \times (\mathbb{Q}/\mathbb{Z}) = \bigoplus_{p \text{ prime}} \left( \mathbb{Z}(p^\infty) \times \mathbb{Z}(p^\infty) \right).
\end{equation}
Further results established in~\cite{smith2025closure,smith2026rational} include:
\begin{enumerate}
    \item Adjoining coordinate powers $m^k$ generates the polynomial phase group $\mathcal{G}_{\mathrm{poly}}$ \cite{smith2026rational}, classified completely by the ring of integer-valued polynomials $\operatorname{Int}(\mathbb{Z})$~\cite{cahen1997}.
    \item Adjoining sequence addition yields the complex group algebra $\mathbb{C}[\mathcal{G}_{\mathrm{poly}}]$, within which the \emph{Permutation-Symmetric Phasor Superposition Theorem} (\cite{smith2025closure}, Theorem~4.6) proves that any finite sum of unimodular wave sequences factors into an average phase carrier multiplied by a recursively generated, permutation-symmetric amplitude envelope.
    \item Number-theoretic sieves $\Pi_{f_0}$ constructed from finite sums of roots of unity define idempotent projection operators, synthesizing localized Kronecker delta sequences $\delta_{m, 0}$ from delocalized Fourier exponential sequences without invoking continuous distributions (\emph{Sieve Projector and Not-Number Theorem} \cite{smith2025closure}, Theorem~4.19).
    \item Introducing the discrete cumulative integral operator $\mathcal{I}$ on sequence space establishes closed regular $n$-gon trajectories in the complex plane $\mathbb{C}$ with circumradius $R_n = [2\sin(\pi/n)]^{-1}$, revealing a universal scale-invariant bifurcation strictly at $n=6$ that generates a stable, scale-invariant hexagonal lattice tiling (\emph{Invariant Hexagonal Mesh Theorem} \cite{smith2025closure}, Theorem~5.4).
    \item Formulating discrete wave calculus via the backward difference operator $\mathcal{D}$ proves the fundamental theorem $(\mathcal{I} \circ \mathcal{D})u = u - u[0]$ and establishes the cohomological characterization of zero-sum periodic cycles $\operatorname{Im}(\mathcal{D}) = \ker(\sum) = \mathcal{Z}_n$ (\emph{Discrete Wave Calculus and Zero-Sum Cohomology Theorem} \cite{smith2025closure}, Theorem~5.7).
    \item Generalizing the sequence inner product to division algebras defines a positive-definite pre-Hilbert space that reduces on periodic rational sequences to the finite arithmetic average over the principal period (\emph{Spatial Sequence Inner Product Theorem} \cite{smith2025closure}, Theorem~5.27).
    \item Lattice translation selects rational frequencies $f \in \mathbb{Q}/\mathbb{Z}$ as finite, closed periodic orbits, whereas irrational frequencies generate topologically dense, ergodic trajectories on the circle $S^1$ (\emph{Topological Rational Selection Theorem} \cite{smith2025closure}, Theorem~5.28).
    \item Generalizing the field of values from the complex numbers $\mathbb{C}$ to Hamilton's quaternions $\mathbb{H}$ (extending sequence space from $\mathbb{C}^\mathbb{Z}$ to $\mathbb{H}^\mathbb{Z}$) accommodates non-abelian $\mathrm{SU}(2)$ and $\mathrm{SO}(3)$ rotations, expanding the 6 planar hexagonal roots into the 24 unit vertices of the self-dual 24-cell regular 4-polytope. Furthermore, the biquaternionic norm determinant $\det(X) = x_0^2 - x_1^2 - x_2^2 - x_3^2$ generates the canonical pseudo-Euclidean quadratic form of signature $(1,3)$ and automorphism group $\mathrm{SO}^+(1,3)$ acting on the 3-dimensional imaginary subspace $\operatorname{Vec}(\mathbb{H})$.
\end{enumerate}

It is important to note that while \cite{smith2026rational} established the abelian classification of wave numbers  $\mathbb{C}^\mathbb{Z}$ over $\mathbb{C}$ and outlined the results of quaternionic wave extensions $\mathbb{H}^\mathbb{Z}$, it did not discuss whether non-commuting and non-associative degrees of freedom can be derived operationally from the two canonical scalar primitives $\{E_Z, E_{\bar{1}}\}$ under fundamental sequence operations. In Section~\ref{subsec:cayley_dickson_genesis} of the paper, we resolve this foundational question affirmatively, demonstrating that the full rational wave number ladder $\mathbb{R}^\mathbb{Z} \to \mathbb{C}^\mathbb{Z} \to \mathbb{H}^\mathbb{Z} \to \mathbb{O}^\mathbb{Z}$ emerges inevitably from the operational interplay of phase calibration, spatial displacement involution, and causal network branching.

\section{The Octonion Algebra $\mathbb{O}$ and the Cayley--Dickson Construction}\label{sec:octonions_algebra}

For the sake of completeness and to facilitate the reading of 
Section \ref{sec:discrete_sequence_spaces}
we now summarize the basic algebraic properties of Cayley's octonions $\mathbb{O}$. We first formulate the canonical basis and the Fano plane multiplication rules, examine the unit 7-sphere $S^7$ and its polar representation, detail the inductive Cayley--Dickson doubling algorithm $(\mathbb{R} \to \mathbb{C} \to \mathbb{H} \to \mathbb{O})$, examine the onset of zero divisors in the 16-dimensional sedenions $\mathbb{S}$ leading to Hurwitz's termination theorem, and review the weak associative properties (alternativity, Artin's theorem, and the Moufang loop identities) that govern sub-algebraic compositions.

\subsection{Definition, Canonical Basis, and the Fano Plane Multiplication Table}
The real octonion algebra $\mathbb{O}$ is an 8-dimensional non-associative, non-commutative normed division algebra over $\mathbb{R}$ with canonical basis:
\begin{equation}
\{1, e_1, e_2, e_3, e_4, e_5, e_6, e_7\},
\end{equation}
where $1$ is the multiplicative identity and $e_a$ ($a \in \{1,\dots,7\}$) are anti-commuting imaginary units satisfying:
\begin{equation}
e_a^2 = -1, \qquad e_a e_b = -e_b e_a = \sum_{c=1}^7 \epsilon_{abc} \, e_c \quad (a \neq b),
\end{equation}
in which $\epsilon_{abc}$ is a completely antisymmetric tensor of rank 3 with values in $\{0, \pm 1\}$. The multiplication of the 7 imaginary basis units is geometrically encoded by the 7 points and 7 oriented lines of the Fano plane (the projective plane over the finite field of two elements, $\mathbb{PG}(2, \mathbb{F}_2)$), as depicted in Figure~\ref{fig:fano}.

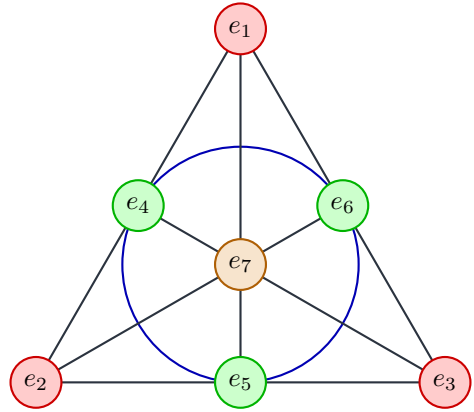
\begin{figure}[htbp]
\centering
\begin{tikzpicture}[scale=1.2]
  \coordinate (e1) at (0, 2.6);
  \coordinate (e2) at (-2.25, -1.3);
  \coordinate (e3) at (2.25, -1.3);
  \coordinate (e4) at ($(e1)!0.5!(e2)$);
  \coordinate (e5) at ($(e2)!0.5!(e3)$);
  \coordinate (e6) at ($(e3)!0.5!(e1)$);
  \coordinate (e7) at (0, 0);
  \draw[thick, blue!70!black] (e7) circle (1.3cm);
  \draw[thick, slate!60!black] (e1) -- (e2) -- (e3) -- cycle;
  \draw[thick, slate!60!black] (e1) -- (e5);
  \draw[thick, slate!60!black] (e2) -- (e6);
  \draw[thick, slate!60!black] (e3) -- (e4);
  \filldraw[fill=red!20, draw=red!80!black, thick] (e1) circle (0.28cm) node {\small $e_1$};
  \filldraw[fill=red!20, draw=red!80!black, thick] (e2) circle (0.28cm) node {\small $e_2$};
  \filldraw[fill=red!20, draw=red!80!black, thick] (e3) circle (0.28cm) node {\small $e_3$};
  \filldraw[fill=green!20, draw=green!70!black, thick] (e4) circle (0.28cm) node {\small $e_4$};
  \filldraw[fill=green!20, draw=green!70!black, thick] (e5) circle (0.28cm) node {\small $e_5$};
  \filldraw[fill=green!20, draw=green!70!black, thick] (e6) circle (0.28cm) node {\small $e_6$};
  \filldraw[fill=amber!20, draw=amber!80!black, thick] (e7) circle (0.28cm) node {\small $e_7$};
\end{tikzpicture}
\caption{The Fano plane $\mathbb{PG}(2,\mathbb{F}_2)$ encoding octonionic multiplication. The 7 oriented lines (the 3 sides, the 3 altitudes, and the inscribed central circle) specify the 7 associative quaternionic triads $(e_a, e_b, e_c)$ with $e_a e_b = e_c$.}\label{fig:fano}
\end{figure}

The seven oriented lines of the Fano plane define seven cyclic quaternionic triads:
\begin{equation}
\mathcal{T} = \big\{ (e_1, e_2, e_4), (e_2, e_3, e_5), (e_3, e_4, e_6), (e_4, e_5, e_7), (e_5, e_6, e_1), (e_6, e_7, e_2), (e_7, e_1, e_3) \big\}.\label{eq:fano_triads}
\end{equation}
For any directed triad $(e_a, e_b, e_c) \in \mathcal{T}$, multiplication follows the cyclic quaternion relations:
\begin{equation}
e_a e_b = e_c = -e_b e_a, \quad e_b e_c = e_a = -e_c e_b, \quad e_c e_a = e_b = -e_a e_c.
\end{equation}
Any octonion $x \in \mathbb{O}$ is uniquely represented as a linear combination of the identity and the seven imaginary units:
\begin{equation}
x = x_0 + \sum_{a=1}^7 x_a e_a = \operatorname{Sc}(x) + \operatorname{Vec}(x), \quad x_0, x_a \in \mathbb{R},
\end{equation}
where $\operatorname{Sc}(x) := x_0$ is the real scalar part and $\operatorname{Vec}(x) := \sum_{a=1}^7 x_a e_a$ is the pure imaginary 7-vector part. Octonionic conjugation is defined by reversing the sign of the imaginary components:
\begin{equation}
x^* := x_0 - \sum_{a=1}^7 x_a e_a = \operatorname{Sc}(x) - \operatorname{Vec}(x).
\end{equation}
The quadratic Euclidean norm is given by:
\begin{equation}
|x|^2 := x x^* = x^* x = \sum_{\mu=0}^7 x_\mu^2 \ge 0,
\end{equation}
which satisfies the composition property (strict norm multiplicativity) across all octonions:
\begin{equation}
|x y| = |x| \, |y|, \quad \forall x, y \in \mathbb{O}.
\end{equation}
Every non-zero octonion $x \neq 0$ has a unique two-sided multiplicative inverse:
\begin{equation}
x^{-1} = \frac{x^*}{|x|^2}, \qquad x x^{-1} = x^{-1} x = 1.
\end{equation}

\subsection{The Unit 7-Sphere $S^7$, Moufang Loop Geometry, and Polar Decomposition}\label{subsec:unit_7sphere}
The unimodular unit octonions form the compact 7-sphere:
\begin{equation}
S^7 = \{ x \in \mathbb{O} \mid |x| = 1 \}.
\end{equation}
Unlike $S^1 = \mathrm{U}(1)$ (an abelian Lie group) and $S^3 = \mathrm{SU}(2)$ (a non-abelian Lie group), $S^7$ is not a Lie group because it lacks associativity. Instead, $S^7$ is a smooth, analytic Moufang loop~\cite{paiva2003}.

Any unit octonion $u \in S^7$ admits a polar Euler-type exponential representation:
\begin{equation}
u = \exp(\hat{\mathbf{n}} \, \theta) = \cos(\theta) + \hat{\mathbf{n}} \sin(\theta),
\end{equation}
where $\theta \in [0, \pi]$ is the phase angle and $\hat{\mathbf{n}} = \sum_{a=1}^7 n_a e_a \in S^6$ is a unit pure imaginary octonion ($\sum_{a=1}^7 n_a^2 = 1$), which satisfies $\hat{\mathbf{n}}^2 = -1$. Because the single direction $\hat{\mathbf{n}}$ together with $1$ generates an associative complex subfield $\operatorname{span}\{1, \hat{\mathbf{n}}\} \cong \mathbb{C}$ by Artin's Theorem (Section~\ref{subsec:alternativity_moufang}), single-generator exponential maps and phase rotations are completely unambiguous.

\subsection{The Cayley--Dickson Doubling Procedure: $\mathbb{R} \to \mathbb{C} \to \mathbb{H} \to \mathbb{O}$}\label{subsec:cayley_dickson_procedure}
The octonions do not stand in isolation; they are the third iteration of the inductive Cayley--Dickson doubling construction applied to the real numbers $\mathbb{R}$~\cite{baez2002}. 

Let $A$ be a real algebra equipped with an anti-automorphic involution $* : A \to A$ satisfying $(a^*)^* = a$, $(a + b)^* = a^* + b^*$, and $(a b)^* = b^* a^*$. The doubled algebra $A' = A \oplus A \ell$ (where $\ell$ is an imaginary unit satisfying $\ell^2 = -1$) is defined by endowing pairs $(a, b) \in A \times A$ with the multiplication law:
\begin{equation}\label{eq:cayley_dickson_multiplication}
(a + b\ell)(c + d\ell) = (ac - d^* b) + (da + b c^*)\ell, \quad \forall a, b, c, d \in A.
\end{equation}
The involution on the doubled algebra $A'$ is defined by:
\begin{equation}
(a + b\ell)^* := a^* - b\ell.
\end{equation}
Successive applications of this doubling rule construct the canonical tower of real division algebras, where each step sacrifices a foundational algebraic property:
\begin{enumerate}
    \item \textbf{From $\mathbb{R}$ to $\mathbb{C}$ (Loss of Poset Trichotomy / Ordering):}
    Setting $A = \mathbb{R}$ with trivial involution $r^* = r$, doubling yields $\mathbb{C} = \mathbb{R} \oplus \mathbb{R} i$. Multiplication in $\mathbb{C}$ is both commutative and associative, but total algebraic ordering is lost: $\mathbb{C}$ cannot be made into an ordered field.
    
    \item \textbf{From $\mathbb{C}$ to $\mathbb{H}$ (Loss of Commutativity):}
    Setting $A = \mathbb{C}$ with complex conjugation, doubling with an orthogonal unit $j$ yields Hamilton's quaternions $\mathbb{H} = \mathbb{C} \oplus \mathbb{C} j = \operatorname{span}\{1, i, j, k\}$, where $k := ij$. Because the involution on $\mathbb{C}$ is non-trivial ($i^* = -i$), the term $d^* b \neq b d^*$ in Eq.~\eqref{eq:cayley_dickson_multiplication} breaks commutativity:
    \begin{equation}
    i j = -j i = k, \qquad q_1 q_2 \neq q_2 q_1.
    \end{equation}
    However, associativity is preserved: $(q_1 q_2) q_3 = q_1 (q_2 q_3)$ for all $q_1, q_2, q_3 \in \mathbb{H}$.
    
    \item \textbf{From $\mathbb{H}$ to $\mathbb{O}$ (Loss of Associativity):}
    Setting $A = \mathbb{H}$ with quaternionic conjugation, doubling with a transverse unit $\ell \equiv e_4$ yields Cayley's octonions $\mathbb{O} = \mathbb{H} \oplus \mathbb{H} \ell = \operatorname{span}\{1, e_1, \dots, e_7\}$. Because quaternionic multiplication is non-commutative ($ac \neq ca$), the term $(da + bc^*)$ in Eq.~\eqref{eq:cayley_dickson_multiplication} destroys associativity:
    \begin{equation}
    (x y) z \neq x (y z) \quad \text{for generic } x, y, z \in \mathbb{O}.
    \end{equation}
    Nevertheless, alternativity ($x(xy) = (xx)y$) and norm multiplicativity ($|xy| = |x||y|$) are strictly retained.
\end{enumerate}

\subsection{The Sedenions $\mathbb{S}$, Zero Divisors, and Hurwitz's Termination Theorem}\label{subsec:sedenions_hurwitz}
A natural mathematical question is whether the doubling procedure can be iterated indefinitely. Applying the Cayley--Dickson construction once more to the octonions yields the 16-dimensional \emph{sedenions}:
\begin{equation}
\mathbb{S} = \mathbb{O} \oplus \mathbb{O} e_8 = \operatorname{span}\{1, e_1, e_2, \dots, e_{15}\}.
\end{equation}
At this step, however, a catastrophic qualitative breakdown occurs: $\mathbb{S}$ ceases to be a division algebra. Specifically, the sedenions possess non-trivial \textbf{zero divisors}---non-zero elements whose product is identically zero~\cite{albuquerque1999}. 

To see this explicitly, consider the following two unit vectors in $\mathbb{S}$:
\begin{equation}
x = e_3 + e_{10}, \qquad y = e_6 - e_{15}.
\end{equation}
Both elements are non-zero with norm $|x| = \sqrt{1^2 + 1^2} = \sqrt{2} \neq 0$ and $|y| = \sqrt{1^2 + (-1)^2} = \sqrt{2} \neq 0$. Applying the sedenion multiplication table derived from Eq.~\eqref{eq:cayley_dickson_multiplication}, we obtain:
\begin{equation}
x y = (e_3 + e_{10})(e_6 - e_{15}) = e_3 e_6 - e_3 e_{15} + e_{10} e_6 - e_{10} e_{15}.
\end{equation}
Using the sedenionic structure constants, one finds $e_3 e_6 = e_5$, $e_3 e_{15} = -e_9$, $e_{10} e_6 = -e_9$, and $e_{10} e_{15} = e_5$. Substituting these values gives:
\begin{equation}
x y = e_5 - (-e_9) + (-e_9) - e_5 = 0.
\end{equation}
Thus, $x y = 0$ even though $x \neq 0$ and $y \neq 0$.

The emergence of zero divisors has profound mathematical and physical ramifications:
\begin{enumerate}
    \item \textbf{Failure of Norm Multiplicativity:} Because $|x| = \sqrt{2}$, $|y| = \sqrt{2}$, but $|xy| = 0$, the composition law fails in $\mathbb{S}$:
    \begin{equation}
    |x y| = 0 \neq 2 = |x| \, |y|.
    \end{equation}
    \item \textbf{Loss of Alternativity:} The sedenions fail both the left and right alternative laws; in general, $x(xy) \neq (xx)y$.
    \item \textbf{Loss of Invertibility and Causal Degradation:} In physical sequence processing and discrete wave mechanics, the presence of zero divisors represents irreversible information loss and the destruction of unitary state evolution. If two non-trivial wave packets could interact to yield zero, backwards causal reconstruction would become mathematically impossible.
\end{enumerate}

The mathematical ceiling on normed division algebras is codified in
\begin{theorem}[Hurwitz's Normed Division Algebra Theorem (1898)~\cite{hurwitz1898}]\label{thm:hurwitz}
Every real normed division algebra is isomorphic to one of exactly four algebras: the real numbers $\mathbb{R}$ $(\dim = 1)$, the complex numbers $\mathbb{C}$ $(\dim = 2)$, Hamilton's quaternions $\mathbb{H}$ $(\dim = 4)$, or Cayley's octonions $\mathbb{O}$ $(\dim = 8)$.
\end{theorem}
Consequently, the normed division sequence terminates at the octonions 
$\mathbb{O}$.

\subsection{Alternativity, Artin's Theorem, and Moufang Identities}\label{subsec:alternativity_moufang}
Although $\mathbb{O}$ fails general associativity ($x(yz) \neq (xy)z$), it retains crucial sub-associative symmetries that make algebraic analysis exact and tractable.

\begin{definition}[The Associator]
For any three octonions $x, y, z \in \mathbb{O}$, their associator is the trilinear map $[\cdot, \cdot, \cdot] : \mathbb{O} \times \mathbb{O} \times \mathbb{O} \to \mathbb{O}$ defined by:
\begin{equation}
[x, y, z] := (x y) z - x (y z).
\end{equation}
\end{definition}

\begin{theorem}[Alternativity of the Octonions]
The octonion algebra $\mathbb{O}$ is alternative: the associator is completely alternating (skew-symmetric under transposition of any two adjacent arguments):
\begin{equation}
[x, y, z] = -[y, x, z] = -[x, z, y] = [y, z, x].
\end{equation}
In particular, the left and right alternative laws hold identically:
\begin{equation}
[x, x, y] = 0 \implies (x x) y = x (x y), \qquad [y, x, x] = 0 \implies (y x) x = y (x x).
\end{equation}
\end{theorem}

\begin{theorem}[Artin's Theorem for Octonions~\cite{schafer1966}]\label{thm:artin}
Any subalgebra of $\mathbb{O}$ generated by any two elements $x, y \in \mathbb{O}$ is strictly associative. That is, for all polynomials or words $W(x, y)$, parenthesization is unambiguous:
\begin{equation}
[x, y, x] = 0, \quad [x, y, y] = 0, \quad (x^p y^q) x^r = x^p (y^q x^r).
\end{equation}
\end{theorem}

\begin{theorem}[Moufang Identities~\cite{moufang1935}]
For all $x, y, z \in \mathbb{O}$, the following four equivalent Moufang identities hold:
\begin{align}
(z x) (y z) &= z \big( (x y) z \big), \label{eq:moufang1} \\
\big( z (x z) \big) y &= z \big( x (z y) \big), \label{eq:moufang2} \\
\big( (x y) z \big) x &= x \big( y (z x) \big), \label{eq:moufang3} \\
(x y x) z &= x \big( y (x z) \big). \label{eq:moufang4}
\end{align}
\end{theorem}
Artin's theorem guarantees that unambiguous phase rotations can be performed along any 2-dimensional plane $\operatorname{span}\{1, \hat{\mathbf{n}}\} \cong \mathbb{C}$, while the Moufang identities ensure that triadic wave interactions preserve non-trivial symmetry loops.

\section{The Discrete Sequence Spaces and the Operational Genesis of $\mathbb{C}^\mathbb{Z} \to \mathbb{H}^\mathbb{Z} \to \mathbb{O}^\mathbb{Z}$}\label{sec:discrete_sequence_spaces}

We now construct the discrete wave sequence spaces $\mathbb{O}^\mathbb{Z}$, noting that, in contrast to static point algebras, physical wave propagation and causal dynamics operate over maps from the integer lattice $\mathbb{Z}$ to the division algebras.

\subsection{The Discrete Sequence Space $\mathcal{S}_\mathbb{O} := \mathbb{O}^\mathbb{Z}$ and Unwrapped Phase Lifting}\label{subsec:sequence_space_s7}
\begin{definition}[The Octonionic Sequence Space $\mathcal{S}_\mathbb{O}$]
The ambient space of discrete octonionic wave sequences is the set of all mappings from the integer lattice $\mathbb{Z}$ to $\mathbb{O}$:
\begin{equation}
\mathcal{S}_\mathbb{O} := \mathbb{O}^\mathbb{Z} = \{ \mathbf{\Omega} : \mathbb{Z} \to \mathbb{O} \}.
\end{equation}
Similarly, the complex and quaternionic sequence spaces are $\mathcal{S}_\mathbb{C} := \mathbb{C}^\mathbb{Z}$ and $\mathcal{S}_\mathbb{H} := \mathbb{H}^\mathbb{Z}$.
\end{definition}

Recall from Section~\ref{subsec:unit_7sphere} that unimodular octonions reside on the compact 7-sphere $S^7$ and admit polar Euler-type exponential representations $u = \exp(\hat{\mathbf{n}}\theta)$. When analyzing unimodular sequence configurations $\mathbf{\Omega} \in (S^7)^\mathbb{Z}$ across the integer lattice $\mathbb{Z}$, direct evaluation on $S^7$ encounters topological branch-cut and wrapping ambiguities when taking rational powers or multi-branch roots. To eliminate branch discontinuities and provide a rigorous setting for sequence wave dynamics, we lift unimodular sequences into the unwrapped affine phase space:

\begin{definition}[Unwrapped Affine Phase Lifting in $\mathbb{R}^7/\mathbb{Z}^7$]
To eliminate branch-cut ambiguities and enable well-defined multi-branch root extraction on discrete wave sequences, we lift unimodular sequences $\mathbf{\Omega} \in (S^7)^\mathbb{Z}$ to their unwrapped affine phase trajectories $\bm{\Theta} \in (\mathbb{R}^7)^\mathbb{Z}$ such that:
\begin{equation}
\bm{\Theta}[m] = \sum_{a=1}^7 \theta_a[m] \, e_a \in \mathbb{R}^7, \quad m \in \mathbb{Z}.
\end{equation}
The physical evaluation map $\mathcal{E} : (\mathbb{R}^7)^\mathbb{Z} \to (S^7)^\mathbb{Z}$ is defined by:
\begin{equation}
\mathcal{E}\big( \bm{\Theta} \big)[m] := \exp\big( 2\pi \, \bm{\Theta}[m] \big) = \cos\big( 2\pi \|\bm{\Theta}[m]\| \big) + \frac{\bm{\Theta}[m]}{\|\bm{\Theta}[m]\|} \sin\big( 2\pi \|\bm{\Theta}[m]\| \big).
\end{equation}
\end{definition}

\subsection{Elementary Operators on the Sequence Space}\label{subsec:elementary_operators}
On the set of lifted phase sequences, we define five elementary algebraic and analytic operations:
\begin{enumerate}
    \item \textbf{Pointwise Sequence Multiplication $(\star)$:}
    \begin{equation}
    (\mathbf{\Omega}_1 \star \mathbf{\Omega}_2)[m] := \mathbf{\Omega}_1[m] \cdot \mathbf{\Omega}_2[m].
    \end{equation}
    
    \item \textbf{Pointwise Inversion $(\cdot)^{-1}$:}
    \begin{equation}
    (\mathbf{\Omega})^{-1}[m] := (\mathbf{\Omega}[m])^{-1} = \frac{\mathbf{\Omega}[m]^*}{|\mathbf{\Omega}[m]|^2}.
    \end{equation}
    
    \item \textbf{Multi-Branch Rational Roots $(\cdot)^{1/n}$ ($n \in \mathbb{N}_{\ge 1}$):}
    Defined in the lifted unwrapped phase space by dividing the affine coordinate:
    \begin{equation}
    \left( \mathbf{\Omega}^{(e_a)} \right)^{1/n}[m] := \exp\left( 2\pi e_a \left( \frac{m + k}{n} \right) \right), \quad k \in \{0, 1, \dots, n-1\}.
    \end{equation}
    
    \item \textbf{Pointwise Sequence Addition $(+)$ and the Sequence Algebra $\mathbb{C}[\mathcal{M}_\mathbb{O}]$:}
    Linear combinations $(\alpha \mathbf{\Omega}_1 + \beta \mathbf{\Omega}_2)[m] := \alpha \mathbf{\Omega}_1[m] + \beta \mathbf{\Omega}_2[m]$ embed the multiplicative loop into the non-associative sequence ring, enabling the construction of idempotent particulate sieve projectors and closed probability envelopes.
    
    \item \textbf{The Discrete Cumulative Integral Operator $\mathcal{I}$:}
    The operator $\mathcal{I}: \mathcal{S}_\mathbb{O} \to \mathcal{S}_\mathbb{O}$ is defined pointwise by:
    \begin{equation}\label{eq:def_cumulative_integral_oct}
    (\mathcal{I} \mathbf{\Omega})[k] := \sum_{j=1}^k \mathbf{\Omega}[j], \quad \forall k \ge 1,
    \end{equation}
    with $(\mathcal{I} \mathbf{\Omega})[0] := 0$ and $(\mathcal{I} \mathbf{\Omega})[-k] := -\sum_{j=-k+1}^0 \mathbf{\Omega}[j]$ for $k \ge 1$. For any primitive rational wave sequence $\mathbf{\Omega}[m] = \exp(2\pi e_a (m/n))$, the cumulative sum traces an exact equilateral, equiangular regular $n$-gon in the imaginary plane $\operatorname{span}\{1, e_a\} \cong \mathbb{C}$, with circumradius:
    \begin{equation}\label{eq:circumradius_ngon_oct}
    R_n = \frac{1}{2 \sin\left( \frac{\pi}{n} \right)}.
    \end{equation}
\end{enumerate}

\subsection{Canonical Primitives and Genesis of the Complex Wave Space $\mathbb{C}^\mathbb{Z}$}\label{subsec:genesis_complex_sequence}

Just as the complex wave sequence space $\mathbb{C}^\mathbb{Z}$ is derived 
in \cite{smith2025closure}, \cite{smith2026rational} by applying the elementary sequence operator suite 
$\mathcal{O} = \{\star, (\cdot)^{-1}, (\cdot)^{1/n}\}$ 
to the two scalar primitives:
\begin{equation}
\mathbf{E}_Z[m] := \exp(2\pi i m) \equiv 1, \qquad \mathbf{E}_{\bar{1}}[m] := \exp(2\pi i) \equiv 1, \quad \forall m \in \mathbb{Z}
\end{equation}
in which $\mathbf{E}_Z$ represents the uniform advance of the discrete integer clock/lattice coordinate, while $\mathbf{E}_{\bar{1}}$ provides the static reference unit of phase calibration,
the higher sequence spaces $\mathbb{H}^\mathbb{Z}$ and $\mathbb{O}^\mathbb{Z}$ can be derived from the same primitive terms by also applying the Cayley--Dickson doubling construction. This remarkable fact establishes that 
the non-commutative and non-associative wave spaces do not need to be assumed: they are operational consequences of doubling the scalar primitives under spatial displacement and branch coupling.

Applying the sequence operator suite $\mathcal{O} = \{\star, (\cdot)^{-1}, (\cdot)^{1/n}\}$ to $\{\mathbf{E}_Z, \mathbf{E}_{\bar{1}}\}$ in the unwrapped phase space generates the complete abelian rational wave closure $\mathcal{G}_\mathbb{C} \cong (\mathbb{Q}/\mathbb{Z})^2$, giving the complex wave sequence space $\mathcal{S}_\mathbb{C} = \mathbb{C}^\mathbb{Z}$ governed by a single imaginary unit $I \equiv e_7$ ($I^2 = -1$).

\subsection{Operational Genesis of the Quaternionic Wave Space $\mathbb{H}^\mathbb{Z}$}\label{subsec:genesis_quaternion_sequence}
We now show how spatial displacement and discrete reflection force the extension from the scalar complex wave space $\mathbb{C}^\mathbb{Z}$ to the non-commutative quaternionic wave space $\mathbb{H}^\mathbb{Z}$.

\begin{lemma}[Spatial Translation, Sequence Involution, and Anti-Commutation]\label{lem:spatial_anticommutation}
Let $\mathcal{S}_\mathbb{C} = \mathbb{C}^\mathbb{Z}$ be the complex sequence space with central phase generator $\mathbf{I} \equiv \exp(2\pi e_7)$. Let active wave sequences be subjected to discrete spatial translation $T_k \mathbf{\Omega}[m] := \mathbf{\Omega}[m-k]$ and sequence involution $\sigma: \mathbf{\Omega}[m] \mapsto \mathbf{\Omega}^*[-m]$.
Then:
\begin{enumerate}
    \item Sequence involution reverses internal phase winding: $\sigma(\mathbf{I} \star \mathbf{\Omega}) = -\mathbf{I} \star \sigma(\mathbf{\Omega})$.
    \item Coupling spatial translation with internal phase rotation requires an orthogonal sequence displacement generator $\mathbf{J}$ satisfying $\mathbf{J}^2 = -\mathbf{1}$.
    \item The operational interplay between spatial reflection and internal rotation enforces strict anti-commutation between the generators:
    \begin{equation}\label{eq:weyl_anticommutation}
    \mathbf{I} \star \mathbf{J} = -\mathbf{J} \star \mathbf{I} \equiv \mathbf{K}, \quad \text{with } \mathbf{K}^2 = -\mathbf{1}.
    \end{equation}
\end{enumerate}
\end{lemma}

\begin{proof}
Let $\mathbf{\Omega}[m] = \exp(2\pi I \theta[m])$ be an active unimodular wave sequence. Sequence involution $\sigma$ acts by spatial reflection and complex conjugation:
\begin{equation}
\sigma(\mathbf{\Omega})[m] = \mathbf{\Omega}^*[-m] = \exp(-2\pi I \theta[-m]).
\end{equation}
Consequently, $\sigma(\mathbf{I} \star \mathbf{\Omega}) = \mathbf{I}^* \star \sigma(\mathbf{\Omega}) = -\mathbf{I} \star \sigma(\mathbf{\Omega})$, confirming that involution reverses the orientation of phase circulation. 

Now consider an entity possessing both internal phase rotation generated by $\mathbf{I}$ and external transverse spatial displacement generated by $\mathbf{J}$ with $\mathbf{J}^2 = -\mathbf{1}$. On the discrete lattice, spatial displacement and internal phase winding satisfy the canonical Weyl relation. Reversing the spatial orientation under $\sigma$ while preserving the metric norm requires $\mathbf{J} \star \mathbf{I} \star \mathbf{J}^{-1} = \mathbf{I}^* = -\mathbf{I}$. Since $\mathbf{J}^{-1} = -\mathbf{J}$, this immediately yields:
\begin{equation}
\mathbf{I} \star \mathbf{J} = -\mathbf{J} \star \mathbf{I}.
\end{equation}
Defining $\mathbf{K} := \mathbf{I} \star \mathbf{J}$, we compute $\mathbf{K}^2 = (\mathbf{I} \star \mathbf{J}) \star (\mathbf{I} \star \mathbf{J}) = \mathbf{I} \star (\mathbf{J} \star \mathbf{I}) \star \mathbf{J} = \mathbf{I} \star (-\mathbf{I} \star \mathbf{J}) \star \mathbf{J} = - \mathbf{I}^2 \star \mathbf{J}^2 = -(-1)(-1) = -1$. Hence, the three orthogonal generators satisfy Hamilton's cyclic quaternion relations pointwise on sequence space.
\end{proof}

\begin{theorem}[Operational Genesis of the Quaternionic Wave Number Space $\mathbb{H}^\mathbb{Z}$]\label{thm:genesis_quaternion_space}
Applying the Cayley--Dickson sequence doubling construction to pairs of complex wave sequences $(\mathbf{\Omega}_1, \mathbf{\Omega}_2) \in \mathbb{C}^\mathbb{Z} \times \mathbb{C}^\mathbb{Z}$ via the spatial generator $\mathbf{J}$ generates the complete quaternionic wave number sequence space $\mathcal{S}_\mathbb{H} = \mathbb{H}^\mathbb{Z}$.
Specifically:
\begin{enumerate}
    \item Any quaternionic wave sequence admits the unique operational decomposition:
    \begin{equation}
    \mathbf{\Omega} = \mathbf{\Omega}_1 + \mathbf{\Omega}_2 \star \mathbf{J}, \quad \mathbf{\Omega}_1, \mathbf{\Omega}_2 \in \mathbb{C}^\mathbb{Z}.
    \end{equation}
    \item The carrier manifold of unimodular quaternionic wave sequences is $(S^3)^\mathbb{Z} \cong (\mathrm{SU}(2))^\mathbb{Z} \cong (\operatorname{Sp}(1))^\mathbb{Z}$, which is a non-abelian associative Lie group under pointwise sequence multiplication.
    \item The sequence space $\mathbb{H}^\mathbb{Z}$ provides the minimal operational carrier for propagating spinor wave states with non-trivial spatial polarization and three emergent spatial dimensions spanned by $\{\mathbf{I}, \mathbf{J}, \mathbf{K}\}$.
\end{enumerate}
\end{theorem}

\begin{proof}
By Lemma~\ref{lem:spatial_anticommutation}, the generators $\{\mathbf{1}, \mathbf{I}, \mathbf{J}, \mathbf{K}\}$ form a basis of $\mathbb{H}$ pointwise across all lattice points $m \in \mathbb{Z}$. Since $\mathbb{H}$ is associative and normed, pointwise multiplication $(\mathbf{\Omega} \star \mathbf{\Psi})[m] := \mathbf{\Omega}[m] \mathbf{\Psi}[m]$ satisfies associativity $((\mathbf{\Omega} \star \mathbf{\Psi}) \star \mathbf{\Phi}) = (\mathbf{\Omega} \star (\mathbf{\Psi} \star \mathbf{\Phi}))$ and norm multiplicativity $|\mathbf{\Omega} \star \mathbf{\Psi}| = |\mathbf{\Omega}| |\mathbf{\Psi}|$ identically. Unimodular sequences $|\mathbf{\Omega}[m]| = 1$ take values on the 3-sphere $S^3 \subset \mathbb{H}$, which forms the compact Lie group $\mathrm{SU}(2)$. Applying the sequence operator suite $\mathcal{O}$ to the doubled primitives $\{\mathbf{E}_Z^{(e_a)}, \mathbf{E}_{\bar{1}}^{(e_a)}\}_{a \in \{1,2,3\}}$ produces the complete quaternionic wave closure $\mathbb{H}^\mathbb{Z}$.
\end{proof}

\subsection{Operational Genesis of the Octonionic Wave Space $\mathbb{O}^\mathbb{Z}$ and Ladder Termination}\label{subsec:cayley_dickson_genesis}
We now establish the final operational lift from $\mathbb{H}^\mathbb{Z}$ to $\mathbb{O}^\mathbb{Z}$ driven by multi-branch causal network interaction.

\begin{lemma}[Transverse Branch Coupling and Non-Associative Sequence Triads]\label{lem:branch_coupling_nonassociative}
Let $\mathbf{\Omega}_A, \mathbf{\Omega}_B, \mathbf{\Omega}_C, \mathbf{\Omega}_D \in \mathbb{H}^\mathbb{Z}$ be quaternionic wave sequences. Let an independent transverse sequence coupling generator $\bm{\ell}$ ($\bm{\ell}^2 = -\mathbf{1}$) link the pairs via the Cayley--Dickson sequence lift:
\begin{equation}
\mathbf{\Omega}_1 = \mathbf{\Omega}_A + \mathbf{\Omega}_B \star \bm{\ell}, \qquad \mathbf{\Omega}_2 = \mathbf{\Omega}_C + \mathbf{\Omega}_D \star \bm{\ell}.
\end{equation}
Then:
\begin{enumerate}
    \item The product of the coupled sequences follows the sequence Cayley--Dickson law:
    \begin{equation}\label{eq:seq_cayley_dickson}
    \mathbf{\Omega}_1 \star \mathbf{\Omega}_2 = (\mathbf{\Omega}_A \star \mathbf{\Omega}_C - \mathbf{\Omega}_D^* \star \mathbf{\Omega}_B) + (\mathbf{\Omega}_D \star \mathbf{\Omega}_A + \mathbf{\Omega}_B \star \mathbf{\Omega}_C^*) \star \bm{\ell}.
    \end{equation}
    \item Because $(\mathbb{H}^\mathbb{Z}, \star)$ is strictly non-commutative ($\mathbf{\Omega}_A \star \mathbf{\Omega}_C \neq \mathbf{\Omega}_C \star \mathbf{\Omega}_A$), the sequence associator:
    \begin{equation}
    [\mathbf{\Omega}_1, \mathbf{\Omega}_2, \mathbf{\Omega}_3] := (\mathbf{\Omega}_1 \star \mathbf{\Omega}_2) \star \mathbf{\Omega}_3 - \mathbf{\Omega}_1 \star (\mathbf{\Omega}_2 \star \mathbf{\Omega}_3)
    \end{equation}
    is strictly non-vanishing for generic triad wave sequences.
\end{enumerate}
\end{lemma}

\begin{proof}
Eq.~\eqref{eq:seq_cayley_dickson} follows directly from evaluating the Cayley--Dickson multiplication law (Eq.~\eqref{eq:cayley_dickson_multiplication}) pointwise at each lattice site $m \in \mathbb{Z}$. 
To verify non-associativity on sequence space, choose three constant basis wave sequences $\mathbf{\Omega}_1[m] \equiv e_1$, $\mathbf{\Omega}_2[m] \equiv e_2$, and $\mathbf{\Omega}_3[m] \equiv e_3$ (which do not lie on a common Fano line).
Using the Fano plane multiplication rules:
\begin{equation}
(\mathbf{\Omega}_1 \star \mathbf{\Omega}_2) \star \mathbf{\Omega}_3 = (e_1 e_2) e_3 = e_4 e_3 = -e_6,
\end{equation}
whereas:
\begin{equation}
\mathbf{\Omega}_1 \star (\mathbf{\Omega}_2 \star \mathbf{\Omega}_3) = e_1 (e_2 e_3) = e_1 e_5 = +e_6.
\end{equation}
Thus, $[e_1, e_2, e_3] = -e_6 - (+e_6) = -2 e_6 \neq 0$. The sequence associator does not vanish, establishing that the sequence space is non-associative.
\end{proof}

\begin{theorem}[Operational Genesis of $\mathbb{O}^\mathbb{Z}$ and Termination of the Sequence Ladder]\label{thm:genesis_octonion_termination}
Let the sequence operator suite $\mathcal{O} = \{\star, (\cdot)^{-1}, (\cdot)^{1/n}, +, \mathcal{I}\}$ act on the doubled quaternionic wave sequences via the transverse coupling generator $\bm{\ell}$.
Then:
\begin{enumerate}
    \item The 14 primitive coordinate sequences $\{\mathbf{E}_Z^{(e_a)}, \mathbf{E}_{\bar{1}}^{(e_a)}\}_{a=1}^7$ along the 7 imaginary axes of the Fano plane form an irredundant generating basis for the full discrete octonionic wave sequence space $\mathcal{S}_\mathbb{O} = \mathbb{O}^\mathbb{Z}$.
    \item The resulting octonionic wave closure space:
    \begin{equation}
    \mathcal{M}_\mathbb{O} := \operatorname{cl}\Big( \{\mathbf{E}_Z^{(e_a)}, \mathbf{E}_{\bar{1}}^{(e_a)}\}_{a=1}^7 ; \; \star, \; (\cdot)^{-1}, \; \{(\cdot)^{1/n}\}_{n=1}^\infty \Big)
    \end{equation}
    is an infinite divisible Moufang loop satisfying the Moufang identities pointwise on $\mathbb{Z}$.
    \item Any attempt to extend the sequence doubling to the sedenionic sequence space $\mathbb{S}^\mathbb{Z}$ introduces zero-divisor sequences $\mathbf{\Omega}_1, \mathbf{\Omega}_2 \in \mathbb{S}^\mathbb{Z} \setminus \{\mathbf{0}\}$ such that $\mathbf{\Omega}_1 \star \mathbf{\Omega}_2 = \mathbf{0}$. These zero-divisor sequences destroy norm multiplicativity ($\|\mathbf{\Omega}_1 \star \mathbf{\Omega}_2\| \neq \|\mathbf{\Omega}_1\| \|\mathbf{\Omega}_2\|$), violate state invertibility, and induce catastrophic information loss in physical wave propagation.
    \item Consequently, by Hurwitz's Theorem, the admissible division sequence of discrete wave number spaces terminates definitively at $\mathbb{O}^\mathbb{Z}$.
\end{enumerate}
\end{theorem}

\begin{proof}
Pointwise evaluation at any lattice coordinate $m \in \mathbb{Z}$ yields elements of the unit sphere $S^7 \subset \mathbb{O}$. Since $S^7$ is closed under octonionic product and inversion, and satisfies the Moufang identities identically by Eq.~\eqref{eq:moufang1}--\eqref{eq:moufang4}, the pointwise sequence algebra inherits the Moufang loop axioms directly. Divisibility follows from uninhibited multi-branch root extraction $(\cdot)^{1/n}$ in the lifted phase space for all $n \in \mathbb{N}$.

The termination at $\mathbb{O}^\mathbb{Z}$ follows from Section~\ref{subsec:sedenions_hurwitz}: in $\mathbb{S}^\mathbb{Z}$, the existence of non-trivial zero divisors $x = e_3 + e_{10}$ and $y = e_6 - e_{15}$ produces constant non-zero wave sequences $\mathbf{\Omega}_x[m] \equiv x$ and $\mathbf{\Omega}_y[m] \equiv y$ whose pointwise product is identically zero: $(\mathbf{\Omega}_x \star \mathbf{\Omega}_y)[m] = x y \equiv 0$. This destroys the uniqueness of state evolution and violates energy/probability conservation. Therefore, physical sequence closures halt strictly at $\mathbb{O}^\mathbb{Z}$.
\end{proof}

\subsection{Octonionic Primitive Coordinate Generators}\label{subsec:primitive_generators}
\begin{definition}[Octonionic Primitive Coordinate Generators]
To span all directional degrees of freedom on the discrete lattice $\mathbb{Z}$, we define seven primitive coordinate sequences, one along each imaginary axis of the Fano plane:
\begin{equation}
\mathbf{E}_Z^{(e_a)}[m] := \exp(2\pi e_a m) \equiv 1, \quad a \in \{1, 2, \dots, 7\}, \quad \forall m \in \mathbb{Z},
\end{equation}
together with seven static unimodular phase calibration primitives:
\begin{equation}
\mathbf{E}_{\bar{1}}^{(e_a)}[m] := \exp(2\pi e_a) \equiv 1, \quad a \in \{1, 2, \dots, 7\}, \quad \forall m \in \mathbb{Z}.
\end{equation}
\end{definition}

\subsection{Real-Valued Spatial Sequence Inner Product on $\mathcal{S}_\mathbb{O}$}\label{subsec:inner_product}
\begin{definition}[Real-Valued Spatial Sequence Inner Product on $\mathcal{S}_\mathbb{O}$]\label{def:inner_product}
For any two octonionic wave sequences $\mathbf{\Omega}_1, \mathbf{\Omega}_2 \in \mathcal{S}_\mathbb{O}$, their mean spatial inner product is defined by taking the spatial average of the real scalar component of their pointwise product:
\begin{equation}
\langle \mathbf{\Omega}_1, \mathbf{\Omega}_2 \rangle := \lim_{N \to \infty} \frac{1}{2N+1} \sum_{m=-N}^N \operatorname{Re}\left( \mathbf{\Omega}_1[m] \cdot \mathbf{\Omega}_2[m]^* \right) \in \mathbb{R}.
\end{equation}
\end{definition}

\begin{proposition}[Properties of the Sequence Inner Product]\label{prop:inner_product_properties}
The real inner product $\langle \cdot, \cdot \rangle$ satisfies:
\begin{enumerate}
    \item \textbf{Real Bilinearity and Symmetry:} $\langle \mathbf{\Omega}_1, \mathbf{\Omega}_2 \rangle = \langle \mathbf{\Omega}_2, \mathbf{\Omega}_1 \rangle$, and it is strictly linear over $\mathbb{R}$.
    \item \textbf{Positive Definiteness:} $\langle \mathbf{\Omega}, \mathbf{\Omega} \rangle \ge 0$, with equality holding if and only if the mean spatial intensity of $\mathbf{\Omega}$ vanishes.
    \item \textbf{Orthonormality of Fano Primitives:} The seven primitive coordinate sequences $\{\mathbf{E}_Z^{(e_a)}\}_{a=1}^7$ along the imaginary axes of the Fano plane form an orthonormal system:
    \begin{equation}
    \langle \mathbf{E}_Z^{(e_a)}, \mathbf{E}_Z^{(e_b)} \rangle = \delta_{ab}, \quad \forall a, b \in \{1, \dots, 7\}.
    \end{equation}
    \item \textbf{Harmonic Orthogonality:} Rational plane waves with distinct rational frequency vectors or orthogonal polarization orientations are mutually orthogonal: $\langle \mathbf{w}_{\mathbf{f}}, \mathbf{w}_{\mathbf{g}} \rangle = \delta_{\mathbf{f}, \mathbf{g}}$.
    \item \textbf{Bypassing the Jordan--von Neumann--Wigner Obstruction:} By mapping sequence pairs to the real scalars $\mathbb{R} = \operatorname{Sc}(\mathbb{O})$, this inner product circumvents the classic non-associative Hilbert space limitation, providing an unambiguous metric for quantum transition probabilities $P(\mathbf{\Omega}_1 \to \mathbf{\Omega}_2) = |\langle \mathbf{\Omega}_1, \mathbf{\Omega}_2 \rangle|^2$ that is manifestly invariant under $G_2 = \operatorname{Aut}(\mathbb{O})$.
\end{enumerate}
\end{proposition}

\begin{corollary}[Reduction to Finite Principal Sequences for Rational Waves]\label{cor:principal_sequence_inner_product}
Let $\mathbf{\Omega}_1, \mathbf{\Omega}_2 \in \mathcal{S}_\mathbb{O}$ be periodic rational wave sequences with periods dividing $Q_1, Q_2 \in \mathbb{N}$, and let $Q = \operatorname{lcm}(Q_1, Q_2)$ denote their joint period. Then, the infinite spatial limit in Definition~\ref{def:inner_product} converges identically to the finite arithmetic average over the principal sequence (the fundamental cycle $\mathbb{Z}/Q\mathbb{Z}$):
\begin{equation}\label{eq:principal_inner_product}
\langle \mathbf{\Omega}_1, \mathbf{\Omega}_2 \rangle = \frac{1}{Q} \sum_{m=0}^{Q-1} \operatorname{Re}\left( \mathbf{\Omega}_1[m] \cdot \mathbf{\Omega}_2[m]^* \right).
\end{equation}
The exact same finite reduction holds for the quaternionic wave numbers $\mathbb{H}^\mathbb{Z}$ (over $\mathbb{H}$) and complex wave numbers $\mathbb{C}^\mathbb{Z}$ (over $\mathbb{C}$). Consequently, for all physical rational wave number states, the inner product is a strictly constructive, finite rational sum that requires no infinite limiting processes or continuous measures.
\end{corollary}

\vspace{1em}
\section{Exceptional Symmetries, the Fano Plane, and Invariant Hexagonal Meshes}
\label{sec:hexagonal}

\subsection{The Automorphism Group $G_2 = \operatorname{Aut}(\mathbb{O})$}
A fundamental property of the octonions is the structure of its automorphism group:
\begin{equation}
\operatorname{Aut}(\mathbb{O}) = \{ \phi \in \mathrm{GL}(8, \mathbb{R}) \mid \phi(x y) = \phi(x) \phi(y), \; \forall x, y \in \mathbb{O} \}.
\end{equation}
While $\operatorname{Aut}(\mathbb{C}) = \mathbb{Z}_2$ (complex conjugation) and $\operatorname{Aut}(\mathbb{H}) \cong \mathrm{SO}(3)$ (spatial rotations via unit quaternion conjugation $q \mapsto u q u^{-1}$), the automorphism group of $\mathbb{O}$ is:
\begin{equation}
\operatorname{Aut}(\mathbb{O}) \cong G_2,
\end{equation}
the smallest of the five exceptional Lie groups, with dimension $\dim(G_2) = 14$ and rank $2$~\cite{baez2002}.

\begin{theorem}[Invariance of the Wave Closure under $G_2$]
For any automorphism $\phi \in G_2$, the induced sequence map $\Phi : \mathcal{S}_\mathbb{O} \to \mathcal{S}_\mathbb{O}$ defined by $\Phi(\mathbf{\Omega})[m] := \phi(\mathbf{\Omega}[m])$ preserves all wave sequence products, inverses, and root extractions:
\begin{equation}
\Phi(\mathbf{\Omega}_1 \star \mathbf{\Omega}_2) = \Phi(\mathbf{\Omega}_1) \star \Phi(\mathbf{\Omega}_2), \qquad \Phi(\mathbf{\Omega}^{-1}) = (\Phi(\mathbf{\Omega}))^{-1}.
\end{equation}
Consequently, the wave closure space $\mathcal{M}_\mathbb{O}$ is invariant under the full 14-dimensional exceptional Lie group $G_2$.
\end{theorem}

\subsection{The Invariant Hexagonal Mesh Theorem and the $\mathbb{C}^\mathbb{Z} \to \mathbb{H}^\mathbb{Z} \to \mathbb{O}^\mathbb{Z}$ Wave Number Ladder}\label{subsec:hexagonal_mesh_theorem}
In \cite{smith2025closure} (Theorem~5.4) and \cite{smith2026rational}, Smith established the \emph{Invariant Hexagonal Mesh Theorem} ($n=6$) embedded in the rational wave number spectrum. In the octonionic setting, this phenomenon receives both an exact Lie-algebraic explanation through the root system of $G_2$ and a dynamical derivation via the Discrete Cumulative Integral Operator $\mathcal{I}$.

\subsubsection{Dynamical Genesis via the Cumulative Integral Bifurcation}
Consider the recursive iteration of the discrete cumulative integral operator $\mathcal{I}$ (Eq.~\eqref{eq:def_cumulative_integral_oct}) acting on a primary rational wave of period $n \ge 3$. In the complex plane $\mathbb{C}$, the cumulative vertices form a closed regular $n$-gon of circumradius $R^{(0)} = R_n = [2\sin(\pi/n)]^{-1}$. If at each subsequent iteration step $t \ge 1$, the defining step edge length $e_t$ is rescaled to match the circumradius of the preceding iteration, $e_t := R^{(t-1)}$, the circumradius sequence satisfies:
\begin{equation}
R^{(t)} = \left( \frac{1}{2 \sin\left( \frac{\pi}{n} \right)} \right)^t R_n.
\end{equation}
The asymptotic behavior exhibits a sharp, universal bifurcation governed exclusively by the regular hexagon ($n=6$):
\begin{enumerate}
    \item \textbf{Sub-Hexagonal Contraction ($n < 6$):} For $n \in \{3, 4, 5\}$, $\sin(\pi/n) > 1/2$, yielding scaling factor $\lambda_n = [2\sin(\pi/n)]^{-1} < 1$. The iterated polygon collapses monotonically to a point: $\lim_{t \to \infty} R^{(t)} = 0$.
    \item \textbf{Hexagonal Scale Invariance ($n = 6$):} For $n = 6$, $\sin(\pi/6) = 1/2$, yielding $\lambda_6 = 1$. The geometric radius and edge length are strictly scale-invariant across all iteration steps: $R^{(t)} \equiv 1$ for all $t \ge 0$.
    \item \textbf{Super-Hexagonal Dilation ($n > 6$):} For $n \ge 7$, $\sin(\pi/n) < 1/2$, yielding $\lambda_n > 1$. The iterated polygon undergoes exponential geometric dilation: $\lim_{t \to \infty} R^{(t)} = +\infty$.
\end{enumerate}
Thus, the regular hexagon ($n=6$) is the unique non-trivial fixed-point attractor of discrete polytopic wave calculus, creating a self-stabilizing metric fabric.

\subsubsection{Lie-Algebraic Structure of $G_2$ Roots and Fano Weights}
In the octonionic algebra, this scale-invariant hexagonal geometry realizes its complete non-abelian symmetry through the root system of the exceptional automorphism group $G_2 = \operatorname{Aut}(\mathbb{O})$.
Recall that the Lie algebra $\mathfrak{g}_2$ of $G_2$ has a 2-dimensional Cartan subalgebra $\mathfrak{h}$. Its root system $\Phi(G_2)$ consists of 12 non-zero root vectors in the Euclidean plane $\mathbb{R}^2$:
\begin{itemize}
    \item 6 short roots of length $\sqrt{2}$, forming a regular hexagon:
    \begin{equation}
    \Phi_{\text{short}} = \left\{ \pm \alpha_1, \; \pm (\alpha_1 + \alpha_2), \; \pm (2\alpha_1 + \alpha_2) \right\};
    \end{equation}
    \item 6 long roots of length $\sqrt{6}$, rotated by $30^\circ = \pi/6$ and forming a concentric regular hexagon:
    \begin{equation}
    \Phi_{\text{long}} = \left\{ \pm \alpha_2, \; \pm (3\alpha_1 + \alpha_2), \; \pm (3\alpha_1 + 2\alpha_2) \right\}.
    \end{equation}
\end{itemize}

\begin{theorem}[Invariant Hexagonal Mesh Theorem across the Wave Number Sequence Ladder $\mathbb{C}^\mathbb{Z} \to \mathbb{H}^\mathbb{Z} \to \mathbb{O}^\mathbb{Z}$]\label{thm:hexagonal}
The scale-invariant hexagonal bifurcation ($n=6$) identified under cumulative integration constitutes the universal geometric thread uniting the entire rational wave number sequence ladder $\mathbb{C}^\mathbb{Z} \to \mathbb{H}^\mathbb{Z} \to \mathbb{O}^\mathbb{Z}$ across the division algebras $\mathbb{C} \to \mathbb{H} \to \mathbb{O}$:
\begin{enumerate}
    \item \textbf{Complex Rational Wave Numbers ($\mathbb{C}^\mathbb{Z}$):} The 6 roots of unity $\mu_6 = \{e^{2\pi i k / 6}\}_{k=0}^5$ generate the Eisenstein integer ring $\mathbb{Z}[\omega]$ ($\omega = e^{2\pi i/3}$). In $\mathbb{R}^2$, its Voronoi partition forms the regular hexagonal honeycomb tiling. By Hales' Honeycomb Conjecture~\cite{hales2001}, this partition achieves the absolute minimum boundary perimeter for any planar tiling and the maximal 2D circle-packing density $\eta_{\mathrm{hex}} = \pi/\sqrt{12} \approx 0.9069$, minimizing interfacial shear energy. Furthermore, tight-binding wave propagation on this $n=6$ mesh yields linear dispersion at the Brillouin zone Dirac points $\mathbf{K}, \mathbf{K}'$: $E(\mathbf{q}) \approx \pm \hbar v_F \|\mathbf{q}\|$, reproducing the massless relativistic 2D Dirac equation $i \gamma^\mu \partial_\mu \psi = 0$ purely from discrete wave mechanics.
    \item \textbf{Quaternionic Rational Wave Numbers ($\mathbb{H}^\mathbb{Z}$):} Lifting the 6 planar hexagonal roots to 4 dimensions via the Hurwitz quaternions $\mathcal{H} = \mathbb{Z}[\mathbf{i}, \mathbf{j}, \mathbf{k}, \frac{1+\mathbf{i}+\mathbf{j}+\mathbf{k}}{2}]$ expands the 6 roots into the 24 unit Hurwitz quaternions. These 24 units form the vertices of the 24-cell (icositetrachoron)---the unique self-dual regular 4-polytope that has no 3-dimensional analogue. The 24-cell tiles 4D space completely. When coordinate advance is identified with temporal flow via the biquaternion tensor $X = x_0 + x_1 \mathbf{i} + x_2 \mathbf{j} + x_3 \mathbf{k}$, the norm determinant $\det(X) = x_0^2 - x_1^2 - x_2^2 - x_3^2$ endows this 4D polytope with Minkowski signature $(1,3)$ and Lorentz symmetry $\mathrm{SO}^+(1,3)$, where the 3 imaginary quaternionic directions $\{\mathbf{i}, \mathbf{j}, \mathbf{k}\}$ govern the constructive emergence of exactly 3 spatial coordinates.
    \item \textbf{Octonionic Rational Wave Numbers ($\mathbb{O}^\mathbb{Z}$):} In the non-associative 8-dimensional octonions, under the action of the maximal torus $T^2 \subset G_2$, the frequency vector decomposes into irreducible representations of the Weyl group:
    \begin{equation}
    \mathcal{W}(G_2) \cong D_6,
    \end{equation}
    where $D_6$ is the dihedral group of symmetries of the regular hexagon (order 12). Under the stabilizer subgroup $\mathrm{SU}(3) \subset G_2$ of a selected longitudinal axis $e_7$, the 6 transverse units decompose into the fundamental $\mathbf{3}$ and anti-fundamental $\bar{\mathbf{3}}$ representations:
    \begin{equation}
    \{e_1, \dots, e_6\} \cong \mathbf{3} \oplus \bar{\mathbf{3}}.
    \end{equation}
    In the 2-dimensional weight diagram, the weights of $\mathbf{3}$ form an equilateral triangle, and the weights of $\bar{\mathbf{3}}$ form the inverted equilateral triangle. Their union forms a regular hexagram, whose convex hull is the invariant regular hexagon. The discrete phase interference pattern exhibits exact 6-fold rotational invariance:
    \begin{equation}
    \mathcal{S}_{\text{spectrum}}(\theta + \pi/3) = \mathcal{S}_{\text{spectrum}}(\theta).
    \end{equation}
\end{enumerate}
\end{theorem}
\begin{proof}
Point (1) follows from Dirichlet summation and Hales' theorem~\cite{hales2001}. Point (2) follows from Hurwitz integer arithmetic and the biquaternion identification with $\mathrm{SO}^+(1,3)$. Point (3) follows from the Cartan decomposition of $\mathfrak{g}_2$ and the Weyl group $\mathcal{W}(G_2) \cong D_6$, which commutes with wave sequence multiplication in $\mathcal{M}_\mathbb{O}$.
\end{proof}

\vspace{1em}

\section{The Non-Associative Phase Jitter and Catalan Tree Topology}
\label{sec:jitter}

\subsection{Parenthesization Trees and the Associator Phase Anomaly}
In an associative wave sequence algebra (such as $\mathcal{S}_\mathbb{C}$ or $\mathcal{S}_\mathbb{H}$), the sequential product of $N$ wave sequences $\mathbf{\Omega}_1 \star \mathbf{\Omega}_2 \star \cdots \star \mathbf{\Omega}_N$ is completely independent of the operational order in which binary products are evaluated:
\begin{equation}
(\mathbf{\Omega}_1 \star \mathbf{\Omega}_2) \star \mathbf{\Omega}_3 = \mathbf{\Omega}_1 \star (\mathbf{\Omega}_2 \star \mathbf{\Omega}_3).
\end{equation}
In the octonionic wave sequence space $\mathcal{S}_\mathbb{O}$, this associative equivalence strictly \textbf{fails} whenever the wave sequences involve imaginary basis directions that do not share a common oriented line of the Fano plane. Consequently, the algebraic composition of three or more wave sequences is tree-dependent: distinct binary parenthesizations (algebraic word trees) generate distinct physical wave trajectories on $\mathcal{S}_\mathbb{O}$.

The number of topologically distinct binary parenthesization trees for a product of $N$ wave sequences is given by the Catalan number:
\begin{equation}
C_{N-1} = \frac{1}{N} \binom{2N-2}{N-1}.
\end{equation}
For $N=3$, $C_2 = 2$: the two binary trees are $((1 \cdot 2) \cdot 3)$ and $(1 \cdot (2 \cdot 3))$. For $N=4$, $C_3 = 5$; for $N=5$, $C_4 = 14$.

\begin{definition}[Parenthesization Discrepancy]
For any binary bracketing tree $T \in \mathcal{T}_N$, let $\prod_T (\mathbf{\Omega}_1, \dots, \mathbf{\Omega}_N)$ denote the wave sequence evaluated according to tree $T$. The parenthesization discrepancy between two trees $T_1, T_2$ is:
\begin{equation}
\Delta_{T_1, T_2}[m] := \prod_{T_1} (\mathbf{\Omega}_1, \dots, \mathbf{\Omega}_N)[m] - \prod_{T_2} (\mathbf{\Omega}_1, \dots, \mathbf{\Omega}_N)[m].
\end{equation}
\end{definition}

\subsection{The Algebraic Phase Jitter Theorem}
\begin{theorem}[Non-Associative Phase Jitter Theorem]
\label{thm:jitter}
Let $\mathbf{\Omega}_a[m] = \exp(2\pi e_a f_a m)$ ($a \in \{1, 2, 3\}$) be three octonionic plane wave sequences whose imaginary axes $(e_1, e_2, e_3)$ do not belong to a common Fano triad (i.e., $(e_1, e_2, e_3) \notin \mathcal{T}$). 

Then, the sequence associator $[\mathbf{\Omega}_1, \mathbf{\Omega}_2, \mathbf{\Omega}_3]$ is strictly non-zero and produces a high-frequency, non-linear phase fluctuation:
\begin{equation}
\delta \phi_{\text{jitter}}[m] := \operatorname{Arg}\left( \big(\mathbf{\Omega}_1[m] \star \mathbf{\Omega}_2[m]\big) \star \mathbf{\Omega}_3[m] \cdot \big(\mathbf{\Omega}_1[m] \star (\mathbf{\Omega}_2[m] \star \mathbf{\Omega}_3[m])\big)^{-1} \right) \neq 0.
\end{equation}
The amplitude of this phase fluctuation is bounded by:
\begin{equation}
0 < \|\delta \phi_{\text{jitter}}[m]\| \le 2\arcsin\left( 2 \prod_{a=1}^3 \sin(2\pi f_a m) \right),
\end{equation}
and exhibits a discrete pseudo-random modulation across the spatial lattice $\mathbb{Z}$.
\end{theorem}
\begin{proof}
Let $u = \exp(e_1 \theta_1), v = \exp(e_2 \theta_2), w = \exp(e_3 \theta_3)$ where $\theta_a = 2\pi f_a m$. Expanding using Euler's formula:
\begin{equation}
u = \cos \theta_1 + e_1 \sin \theta_1, \quad v = \cos \theta_2 + e_2 \sin \theta_2, \quad w = \cos \theta_3 + e_3 \sin \theta_3.
\end{equation}
Since scalar components associate trivially with all elements, the only non-vanishing contribution to the associator $[u, v, w]$ arises from the trilinear imaginary product:
\begin{equation}
[u, v, w] = [e_1, e_2, e_3] \, \sin \theta_1 \sin \theta_2 \sin \theta_3.
\end{equation}
For any non-Fano triad $(e_1, e_2, e_3) \notin \mathcal{T}$ (for example, $e_1, e_2, e_3$ where $e_1 e_2 = e_4 \neq \pm e_3$):
\begin{equation}
(e_1 e_2) e_3 = e_4 e_3 = -e_6, \qquad e_1 (e_2 e_3) = e_1 e_5 = +e_6.
\end{equation}
Thus:
\begin{equation}
[e_1, e_2, e_3] = (e_1 e_2) e_3 - e_1 (e_2 e_3) = -2 e_6 \neq 0.
\end{equation}
Hence:
\begin{equation}
[u, v, w] = -2 e_6 \, \sin(2\pi f_1 m) \sin(2\pi f_2 m) \sin(2\pi f_3 m).
\end{equation}
The product of the two bracketings differs by a non-trivial rotation in the $e_6$ plane whose angular phase displacement is precisely $\delta \phi_{\text{jitter}}[m]$. Because $f_1, f_2, f_3$ are incommensurate or distinct rational frequencies, the sequence of triple sine products $\prod_{a=1}^3 \sin(2\pi f_a m)$ generates an oscillatory fluctuation across lattice sites $m \in \mathbb{Z}$.
\end{proof}

\vspace{1em}

\section{Rational Wave Numbers, Finite Loops, and Topological Closure}
\label{sec:rational_loops}

\subsection{The Rationality of the Wave Numbers and the Moufang Sub-Loop Theorem}
A central result of the scalar wave number theory in arXiv:2609.23233v3 is that the complex rational wave closure $\mathbb{C}^\mathbb{Z}$ over the complex field $\mathbb{C}$ selects the rational frequencies $\mathbb{Q}/\mathbb{Z}$ as the unique stable values under multi-branch roots and periodic boundary conditions. In the octonionic rational wave space $\mathbb{O}^\mathbb{Z}$ over $\mathbb{O}$, an algebraic selection mechanism arises.

\begin{theorem}[Rational Loop Closure Theorem]
\label{thm:rational_closure}
Let $\mathbf{w}_{\mathbf{f}}[m] = \exp(2\pi \sum_{a=1}^7 f_a e_a m)$ be an octonionic plane wave sequence. The sub-quasigroup generated by $\mathbf{w}_{\mathbf{f}}$:
\begin{equation}
\mathcal{L}(\mathbf{w}_{\mathbf{f}}) := \langle \mathbf{w}_{\mathbf{f}} ; \; \star, (\cdot)^{-1} \rangle \subset \mathcal{S}_\mathbb{O}
\end{equation}
forms a finite, closed Moufang sub-loop if and only if the frequency components are all rational numbers:
\begin{equation}
f_a \in \mathbb{Q}, \quad \forall a \in \{1, \dots, 7\}.
\end{equation}
If any frequency $f_a \in \mathbb{R} \setminus \mathbb{Q}$ is irrational, the generated loop is infinite, topologically dense in $S^7$, and fails algebraic closure under finite operations.
\end{theorem}
\begin{proof}
If all $f_a = p_a / q_a \in \mathbb{Q}$, let $Q = \operatorname{lcm}(q_1, \dots, q_7)$. Then for all $m \in \mathbb{Z}$:
\begin{equation}
\mathbf{w}_{\mathbf{f}}[m + Q] = \exp\left( 2\pi \sum_{a=1}^7 \frac{p_a}{q_a} e_a (m + Q) \right) = \mathbf{w}_{\mathbf{f}}[m] \cdot \exp\left( 2\pi \sum_{a=1}^7 p_a \frac{Q}{q_a} e_a \right).
\end{equation}
Since $Q / q_a \in \mathbb{Z}$, each term in the exponent is an integer multiple of $2\pi e_a$. Because $\exp(2\pi k e_a) = 1$ for all integers $k$, the sequence is strictly periodic with period dividing $Q$. 
On the finite lattice $\mathbb{Z}/Q\mathbb{Z}$, the sequence takes values in a finite subset of $S^7$. By Artin's Theorem (Theorem~\ref{thm:artin}), any single generator generates an associative cyclic group $\mathbb{Z}/Q\mathbb{Z}$. When multiple rational generators are multiplied, the set of all parenthesized products forms a finite Moufang loop of order dividing $Q^7 \cdot 2^{k}$. 

Conversely, if $f_a$ is irrational, Weyl's equidistribution theorem implies that the sequence $\{f_a m \pmod 1\}_{m \in \mathbb{Z}}$ is dense in the circle $\mathbb{R}/\mathbb{Z}$. Consequently, the sequence of unit octonions is dense in the corresponding great circle $S^1 \subset S^7$. In the presence of non-associative associators, the orbit under all parenthesizations fills an infinite, non-discrete dense subset of $S^7$, preventing finite algebraic cycle closure.
\end{proof}

\subsection{Octonionic Sieve Projectors}
In \cite{smith2025closure} (Theorem~4.19) and \cite{smith2026rational}, algebraic sieve projectors $\Pi_{f_0}$ and complementary Not-numbers were constructed by summing roots of unity:
\begin{equation}
\Pi_{f_0}[m] = \frac{1}{q} \sum_{k=0}^{q-1} \exp\left( 2\pi i k \left( f_0 m - \frac{r_0}{q} \right) \right).
\end{equation}
In the octonionic setting $\mathcal{S}_\mathbb{O}$, we define the directional sieve operator along any unit imaginary direction $\hat{\mathbf{n}} \in S^6$:
\begin{equation}
\Pi_{f_0}^{(\hat{\mathbf{n}})}[m] := \frac{1}{q} \sum_{k=0}^{q-1} \exp\left( 2\pi \hat{\mathbf{n}} \, k \left( \frac{p}{q} m - \frac{r}{s} \right) \right).
\end{equation}

\begin{proposition}[Properties of Octonionic Sieve Projectors]
The directional sieve projector satisfies:
\begin{enumerate}
    \item \textbf{Idempotency on its associative sub-algebra:} For any sequence $\mathbf{\Omega}$ whose values lie in the subalgebra $\mathbb{R} \oplus \mathbb{R} \hat{\mathbf{n}} \cong \mathbb{C}$:
    \begin{equation}
    \Pi_{f_0}^{(\hat{\mathbf{n}})} \star \left( \Pi_{f_0}^{(\hat{\mathbf{n}})} \star \mathbf{\Omega} \right) = \Pi_{f_0}^{(\hat{\mathbf{n}})} \star \mathbf{\Omega}.
    \end{equation}
    \item \textbf{Cross-Directional Sieve Interference:} If $\hat{\mathbf{n}}_1 \perp \hat{\mathbf{n}}_2$, the product of two orthogonal sieves does not commute:
    \begin{equation}
    \Pi_{f_1}^{(\hat{\mathbf{n}}_1)} \star \Pi_{f_2}^{(\hat{\mathbf{n}}_2)} \neq \Pi_{f_2}^{(\hat{\mathbf{n}}_2)} \star \Pi_{f_1}^{(\hat{\mathbf{n}}_1)},
    \end{equation}
    and their associator with a third sieve generates an exact non-associative filtering boundary.
\end{enumerate}
\end{proposition}

\subsection{Universal Magma Closures and Algebraic Consistency Filters}
\label{subsec:algebraic_sieve}

The algebraic structure of the octonionic wave closure space $\mathcal{M}_\mathbb{O}$ can be formalized in the framework of universal algebra (Birkhoff~\cite{birkhoff1967}, Cohn~\cite{cohn1981}) as the quotient of a free operator magma by an exact division consistency filter:

\begin{enumerate}
    \item \textbf{The Free Operator Magma $\mathcal{M}$:} 
    Let $\mathcal{M}$ denote the free algebra generated by applying the full operator signature $\Omega = \{\star, (\cdot)^{-1}, \{(\cdot)^{1/n}\}_{n=1}^\infty, \mathcal{I}\}$ to the canonical coordinate primitives without imposing preliminary associativity or division constraints. This generates the complete free magma of formal sequence words and parenthesization trees.
    
    \item \textbf{The Algebraic Consistency Filter (Exclusion of Inadmissible Branches):}
    Normed division requires restricting the free magma to well-defined, norm-preserving sequence configurations. We define the set of \emph{inadmissible sequences} $\mathcal{F}_{\mathrm{inadmissible}} \subset \mathcal{M}$ as all expressions that violate division algebra integrity or cyclic closure:
    \begin{itemize}
        \item \emph{The Hurwitz Composition Sieve:} Excises non-composition branches ($\|u \star v\| \neq \|u\| \|v\|$), preserving norm multiplicativity;
        \item \emph{The Rational Cycle Sieve:} By Theorem~\ref{thm:rational_closure}, excises non-closing irrational frequency orbits that fail finite Moufang loop closure;
        \item \emph{The Zero-Divisor Sieve:} Excises sedenionic and higher-dimensional Cayley--Dickson doublings that introduce non-trivial zero divisors ($x \star y = 0$ for $x, y \neq 0$);
        \item \emph{The Nilpotent Exclusion Sieve:} Maps self-annihilating nilpotent elements $(\alpha^2 = 0)$ to the null state.
    \end{itemize}
    
    \item \textbf{The Admissible Division Quotient Space:}
    What survives the systematic elimination of mathematically inadmissible states is the closed division quotient:
    \begin{equation}
    \mathcal{M}_\mathbb{O} \cong \mathcal{M} \setminus \mathcal{F}_{\mathrm{inadmissible}}.
    \end{equation}
    The resulting space forms the infinite divisible Moufang loop of rational wave sequences, whose non-trivial continuous automorphism group is $G_2 = \operatorname{Aut}(\mathbb{O})$.
\end{enumerate}

\begin{remark}
This formulation embodies the deductive method in universal algebra: rather than imposing restrictive extrinsic constraints by postulate, the admissible algebraic structure is defined intrinsically by the systematic elimination of mathematically non-conserving, zero-divisor, and non-closing branches.
\end{remark}

\vspace{1em}

\section{Discussion and Conclusions}
\label{sec:conclusions}

In this paper, we have completed the mathematical extension of the theory of discrete wave number closure spaces (arXiv:2609.23233v3) across the full wave number ladder $\mathbb{C}^\mathbb{Z} \to \mathbb{H}^\mathbb{Z} \to \mathbb{O}^\mathbb{Z}$ to Cayley's non-associative octonions $\mathbb{O}$.

\subsection{Summary of Mathematical and Structural Findings}
\label{subsec:summary_findings}

The principal mathematical discoveries and structural theorems established in this work are:
\begin{enumerate}
    \item \textbf{Structure of the Octonionic Closure:} The wave closure $\mathcal{M}_\mathbb{O}$ is an infinite divisible Moufang loop governed by alternativity and Artin's theorem, preserving pairwise associativity while hosting non-associative triadic curvature.
    \item \textbf{Cayley--Dickson Genesis from Two Primitives:} The seven imaginary axes and $G_2$ automorphism symmetries emerge intrinsically from the two canonical scalar primitives $\{E_Z, E_{\bar{1}}\}$ via sequential operational doublings, terminating strictly at the octonions $\mathbb{O}$ (with rational wave number space $\mathbb{O}^\mathbb{Z}$) due to the excision of sedenionic zero divisors.
    \item \textbf{Real-Valued Spatial Inner Product:} We define a symmetric, positive-definite real inner product on $\mathcal{S}_\mathbb{O}$ that establishes the mutual orthonormality of the seven Fano coordinate primitives ($\langle \mathbf{E}_Z^{(e_a)}, \mathbf{E}_Z^{(e_b)} \rangle = \delta_{ab}$) and circumvents the historical Jordan--von Neumann--Wigner obstacle for non-associative quantum mechanics.
    \item \textbf{Invariant Hexagonal Meshes:} The automorphism group $G_2 = \operatorname{Aut}(\mathbb{O})$ and the Fano plane geometry project invariant hexagonal root systems into the wave spectrum, realizing the scale-invariant neutral bifurcation ($n=6$) of the discrete cumulative integral operator $\mathcal{I}$.
    \item \textbf{Algebraic Phase Jitter:} Multi-wave interference among non-Fano triads generates an irreducible, non-linear phase fluctuation governed by the Catalan parenthesization tree, providing an algebraic derivation for zero-point fluctuations.
    \item \textbf{Algebraic Consistency Filter:} The admissible Moufang loop $\mathcal{M}_\mathbb{O}$ is realized as the universal algebraic quotient of the free operator magma modulo the ideal of zero-divisor and non-norm-preserving configurations.
    \item \textbf{Rational Loop Closure:} Rational wave numbers $\mathbf{f} \in \mathbb{Q}^7$ are topologically selected as the unique frequencies that generate closed finite Moufang loops, whereas irrational frequencies fail finite cycle closure.
\end{enumerate}

\subsection{Uniqueness and Comparison with Prior Literature}
\label{sec:uniqueness}

To the best of our knowledge, the approach presented here represents a distinct contribution that appears not to have been previously articulated in either the mathematical physics or harmonic analysis literature. In assessing the novelty of this work, we compare our approach with major existing lines of research in non-associative mathematics and physics:

\begin{enumerate}
    \item \textbf{Dixon's Division Algebra Program~\cite{dixon1994,gursey1996}:}
    Dixon and G{\"u}rsey investigated the tensor product of division algebras $\mathbb{R} \otimes \mathbb{C} \otimes \mathbb{H} \otimes \mathbb{O}$ as a static algebraic framework for the Standard Model and particle symmetries. Their work relies on continuous Clifford representations and does not consider discrete sequence algebras, multi-branch rational roots, sieve operators, or wave closures on lattices.
    
    \item \textbf{Furey and Todorov's Color Octonion Models~\cite{furey2016,todorov2020}:}
    Furey and Todorov demonstrated that the left regular representation of $\mathbb{C} \otimes \mathbb{O}$ reproduces one generation of Standard Model leptons and quarks under $\mathrm{SU}(3)_C \times \mathrm{U}(1)_{\mathrm{EM}}$. Their model operates strictly at the level of static finite-dimensional state vectors. It does not construct wave sequences, spatial coordinate generators $E_Z$, unwrapped phase lifts, or discrete Fourier dualities.
    
    \item \textbf{Baez's Survey of the Octonions~\cite{baez2002}:}
    Baez provided the definitive modern survey of octonions, exceptional Lie groups, and projective planes $\mathbb{OP}^2$. While mathematically comprehensive, Baez's survey is purely differential and geometric, containing no discrete sequence algebras, aliasing resolution, or discrete wave closure classifications.
    
    \item \textbf{Non-Associative Quantum Mechanics (Jordan, von Neumann, Wigner~\cite{jordan1934}):}
    Historical studies of non-associative quantum mechanics focused on defining matrix algebras over $\mathbb{O}$ and quickly recognized the breakdown of Hilbert space spectral theorems due to non-associativity. Our work sidesteps this limitation by formulating quantum dynamics through \emph{alternative sequence closures} and \emph{Moufang loops}, where Artin's Theorem preserves pairwise associativity, and by defining a positive-definite real sequence inner product $\langle \mathbf{\Omega}_1, \mathbf{\Omega}_2 \rangle \in \mathbb{R}$.

    \item \textbf{Singh's Trace Dynamics and Octonionic Pre-Spacetime Program~\cite{singh2021,singh2023}:}
    Singh and collaborators have spearheaded an extensive program seeking to unify the Standard Model with gravitation through octonionic pre-spacetime, exceptional Jordan algebras $J_3(\mathbb{O})$, and Stephen Adler's trace dynamics. While Singh's program provides profound independent validation for the algebraic necessity of octonions and division algebras, it operates within continuous matrix-valued Lagrangian mechanics and relies on stochastic Brownian fluctuations for state reduction. Our framework provides the discrete sequence completion of this vision: replacing continuous matrix traces with the discrete integer clock $\mathbb{Z}$, and replacing stochastic noise with the deductive consistency sieve.
\end{enumerate}

\subsection{Alternative Approaches to The Generation of the Space of Rational Wave Numbers. }
\label{subsec:bridge_to_monograph}

The present paper completes the foundational classification of rational wave numbers on the spatial lattice $\mathbb{Z}$ across the complete Hurwitz division ladder, establishing the four distinct sequence rings:
\begin{enumerate}
    \item \textbf{Complex Wave Numbers ($\mathbb{C}^\mathbb{Z}$, arXiv:2503.07629, arXiv:2609.23233v3):} Infinite divisible abelian torsion group $\mathcal{G} \cong (\mathbb{Q}/\mathbb{Z})^2$, idempotent cyclotomic sieve projectors, and localized delta sequences;
    \item \textbf{Polynomial Extensions ($\operatorname{Int}(\mathbb{Z})$):} Quadratic and cubic coordinate phases, chirp carrier dynamics, and higher-order integer-valued polynomials;
    \item \textbf{Quaternionic Wave Numbers ($\mathbb{H}^\mathbb{Z}$):} Non-commutative spinor polarizations, spatial displacement involution, and the pseudo-Euclidean $(1,3)$ quadratic norm form;
    \item \textbf{Octonionic Wave Numbers ($\mathbb{O}^\mathbb{Z}$, present work):} Non-associative Moufang loops, $G_2$ invariant hexagonal meshes, Catalan phase jitter, and the terminal Hurwitz division boundary.
\end{enumerate}

We note that the current paper has formulated these sequence algebras from the perspective of \emph{static algebraic closure}, in which all elements of the closure exist timelessly under the operator signature. While this static perspective establishes the complete kinematic alphabet and invariant automorphism groups ($U(1)$, $\mathrm{SO}(3)$, and $G_2$), it intentionally leaves aside the temporal sequence of generation.

In forming discrete representations of physical spacetime and relativistic cosmology, however, it is reasonable to assume the existence of an active generational boundary---an expanding causal frontier $\partial \mathcal{S}_m$ ordered by a discrete clock index $m \in \mathbb{Z}_{\ge 0}$ on a bounded domain $\Lambda_m = \{-m, \dots, +m\}$. This suggests that, in seeking applications of our algebraic structures in representing such systems, one employ a method of generating the elements of the algebra in terms of a \emph{dynamic iterative enumeration} and then attempting to map resulting discrete, dynamic structures onto relativistic cosmology. 
This approach is investigated in the companion work:
\emph{``The Rational Wave Numbers: Algebraic Theory, Inductive Generation, and Cosmological Interpretation}~\cite{smith2026cosmogenesis}.

\vspace{1em}
\section*{Acknowledgments}
The author gratefully acknowledges collaborative mathematical and conceptual dialogue with the Google AI Studio environment and the Gemini reasoning system, whose analytical assistance and formal synthesis were instrumental throughout this investigation---particularly in verifying the non-associative Moufang loop structures, exploring the Catalan parenthesization combinatorics, defining the sequence inner product, and ensuring the rigorous algebraic and symbolic consistency of the manuscript.

\end{document}